\documentclass[12pt]{amsart}

\usepackage{palatino}
\usepackage[left=3cm,top=3.8cm,right=3cm]{geometry}        
\usepackage{comment}
\usepackage{amssymb}
\usepackage{enumerate}

\usepackage{xcolor}
\usepackage[pagebackref]{hyperref}
\hypersetup{
   colorlinks,
    linkcolor={red!60!black},
    citecolor={blue!60!black},
    urlcolor={blue!90!black}
}
\DeclareMathOperator{\R}{\mathbb{R}}
\DeclareMathOperator{\N}{\mathbb{N}}
\DeclareMathOperator{\Z}{\mathbb{Z}}

\DeclareMathOperator{\e}{\varepsilon}
\DeclareMathOperator{\DD}{\mathcal{D}}
\DeclareMathOperator{\PP}{\mathcal{P}}

\DeclareMathOperator{\hdim}{dim_H}
\DeclareMathOperator{\pdim}{dim_P}
\DeclareMathOperator{\supp}{supp}

\newtheorem{theorem}{Theorem}[section]
\newtheorem{conjecture}[theorem]{Conjecture}
\newtheorem{lemma}[theorem]{Lemma}
\newtheorem{definition}[theorem]{Definition}
\newtheorem{prop}[theorem]{Proposition}

\numberwithin{equation}{section}

\begin{document}

\title[Furstenberg and Falconer]{Slices and distances: on two problems of Furstenberg and Falconer}

\author{Pablo Shmerkin}

\address{Department of Mathematics, the University of British Columbia, Canada}
\email{pshmerkin@math.ubc.ca}
\urladdr{http://pabloshmerkin.org}
\thanks{This work was partially supported by an NSERC discovery grant and by Project PICT 2015-3675 (ANPCyT)}

\begin{abstract}
We survey the history and recent developments around two decades-old problems that continue to attract a great deal of interest: the slicing $\times 2$, $\times 3$ conjecture of H. Furstenberg in ergodic theory, and the distance set problem in geometric measure theory introduced by K. Falconer. We discuss some of the ideas behind our solution of Furstenberg's slicing conjecture, and recent progress in Falconer's problem. While these two problems are on the surface rather different, we emphasize some common themes in our approach: analyzing fractals through a combinatorial description in terms of ``branching numbers'', and viewing the problems through a ``multiscale projection'' lens.
\end{abstract}

\maketitle


\section{Introduction}

In this article we survey recent progress on the following two old conjectures. Hausdorff dimension is denoted $\hdim$.

\begin{conjecture}[{Furstenberg's slicing conjecture, \cite{Furstenberg70}}] \label{conj:Furstenberg}
Let $X,Y\subset [0,1)$ be closed and invariant under $T_a, T_b$ respectively, where $T_m(x)=mx\bmod 1$ is multiplication by $m$ on the circle. Assume that $\log a/\log b$ is irrational. Then
\[
\hdim( (X\times Y)\cap \ell) \le \max(\hdim(X)+\hdim(Y)-1,0)
\]
for all lines $\ell$ that are neither vertical nor horizontal.
\end{conjecture}

\begin{conjecture}[{Falconer's distance set problem, originating in \cite{Falconer85}}] \label{conj:Falconer}
Let $X\subset\R^d$, $d\ge 2$ be a Borel set with $\hdim(X)\ge d/2$. Let $\Delta(X)=\{ |x-y|:x,y\in X\}$. Then
\[
\hdim(\Delta(X)) = 1.
\]
\end{conjecture}

We discuss the history and motivation behind these conjectures in Sections \ref{sec:Furstenberg} and \ref{sec:Falconer}, respectively. Conjecture \ref{conj:Furstenberg} was resolved by the first author \cite{Shmerkin19} and, simultaneously, independently, and with a strikingly different proof, by M. Wu \cite{Wu19}. Many related problems remain open. Conjecture \ref{conj:Falconer} is open in all dimensions.

At first sight, Conjectures \ref{conj:Furstenberg} and \ref{conj:Falconer} appear to be rather different (other than both involving Hausdorff dimension). A key difference is that Furstenberg's conjecture deals with sets with a rigid arithmetic structure, while Falconer's conjecture involves arbitrary Borel sets. A more subtle but also crucial distinction is that Furstenberg's conjecture is \emph{linear} in nature (it concerns linear slices of $X\times Y$), while Falconer's conjecture deals with Euclidean distances and curvature plays a key r\^{o}le in all partial progress towards it.

Nevertheless, we will see that there are some similar ideas in our own approach to these two problems. We will recast both in terms of \emph{projections}. To handle these projections, we use in both cases a combinatorial approach to the study of fractals through their branching structure. Bourgain's celebrated discretized projection theorem \cite{Bourgain03, Bourgain10} (or its proof) makes an appearance in our work on both conjectures.

In Section \ref{sec:Bourgain}, we discuss a key uniformization lemma, and Bourgain's discretized sumset, sum-product, and projection theorems. In Section \ref{sec:Furstenberg}, we put Furstenberg's slicing conjecture into context and give an impressionistic account of our solution. In Section \ref{sec:Falconer}, we discuss Falconer's problem and some of our recent progress towards it (obtained partly in collaboration with T. Keleti and with H. Wang). Along the way, we will touch upon the closely related and vast field of \emph{projection theory} in geometric measure theory.

A word on notation. Given two positive quantities $A, B$, the notation $A\lesssim B$ means that $A\le C B$ for some constant $C>0$, while $A\lesssim_x B$ means that $A\le C(x) B$, where again $C(x)>0$. We write $A\gtrsim B$ for $B\lesssim A$ and $A\sim B$ for $A\lesssim B\lesssim A$, and likewise with sub-indices. We denote positive constants whose value is not too important by $c, C$ and as before indicate their dependencies with subindices.

\section{A glimpse of Bourgain's discretized geometry}
\label{sec:Bourgain}

\subsection{Uniform sets and uniformization}

Even though the statements of Conjectures \ref{conj:Furstenberg} and \ref{conj:Falconer} involve Hausdorff dimension, most approaches discretize the problem at a small scale $\delta$. Given a bounded set $X\subset\R^d$, let $|X|_{\delta}$ be the number of $\delta$-mesh cubes $\prod_{i=1}^d [k_i\delta,(k_i+1)\delta)$ intersecting $X$. If $X$ has the property that $|X|_{\delta}\le \delta^{-s}$ for all small $\delta$, then $\hdim(X)\le s$. If, on the other hand, $|X|_{\delta}\ge \delta^{-s}$ for all small $\delta$, it does not quite follow that $\hdim(X)\ge s$ - what is technically true is that the lower Minkowski dimension of $X$ is at least $s$. For the sake of simplicity we will ignore this distinction, and consider the growth rate of $|X|_{\delta}$ as a good proxy for the (Hausdorff) dimension of $X$. In this discussion, there is no loss of generality in restricting $\delta$ to dyadic numbers $2^{-m}$ or even $(2^T)$-adic numbers $2^{-T\ell}$ once the integer $T$ has been fixed.

Let $\mathcal{D}_{\delta}$ denote the family of $\delta$-mesh cubes in $\R^d$. If $X\subset\R^d$ is a union of cubes in $\mathcal{D}_{2^{-m}}$, we say that $X$ is a $2^{-m}$-set. 
For $X\subset\R^d$, we denote the set of cubes in $\mathcal{D}_{\delta}$ intersecting $X$ by $\mathcal{D}_{\delta}(X)$.

Let $X\subset [0,1)^d$. Given a $T\in\N$ (which we consider fixed) and $\ell\ge 1$, we can view $(\mathcal{D}_{2^{-Tj}}(X))_{j=0}^{\ell-1}$ as a tree, with $[0,1)^d$ as the root and descendance given by inclusion. This tree provides a combinatorial description of $X$ at resolution $2^{-T\ell}$. In general, the tree may be very irregular, with different vertices having different numbers of offspring. In many situations, the set $X$ is easier to study if one knows that the tree is \emph{spherically symmetric},  meaning that the number of offspring is constant at each level of the tree (but can still change from level to level).
\begin{definition}
A set $X\subset [0,1)^d$ is $(T,(N_j)_{j=0}^{\ell-1})$-uniform if
\[
Q\in\mathcal{D}_{2^{-jT}}(X)\Longrightarrow |\mathcal{D}_{2^{-(j+1)T}}(X\cap Q)| = N_j, \quad j=0,1,\ldots,\ell-1.
\]
If $X$ is $(T,(N_j)_{j=0}^{\ell-1})$-uniform for some $(N_j)_{j=0}^{\ell-1}$, then we also say that $X$ is $(T;\ell)$-uniform.
\end{definition}

We emphasize that what is fixed at each scale is the \emph{number} of offspring; the particular set of $N_j$ sub-cubes is still allowed to depend on the parent cube of level $j$. The following \emph{uniformization lemma} says that by taking $T$ large, and at the price of replacing $X$ by a large subset, we may always assume that $X$ is $(T;\ell)$-uniform.
\begin{lemma} \label{lem:uniformization}
Fix $T,\ell\in\N$ and write $m=T\ell$. Let $X\subset [0,1)^d$ be a $2^{-m}$-set. Then $X$ contains a $(T;\ell)$-uniform subset $X'$ with
\[
|X'|\ge (2T)^{-\ell}|X| = 2^{(-\log(2T)/T) m}|X|.
\]
\end{lemma}
\begin{proof}
We begin from the bottom of the tree, setting $X^{(\ell)}:=X$. Once $X^{(j+1)}$ is constructed, we let
\[
X^{(j,k)} =\bigcup \left\{ X^{(j+1)}\cap Q :  |Q\cap X^{(j+1)}|_{2^{-(j+1)T}} \in [2^k+1,2^{k+1}] \right\},\quad k=0,\ldots,T-1.
\]
Since $k$ takes $T$ values, we can pick $k=k_j$ such that $|X^{(j,k)}|\ge |X^{(j+1)}|/T$. By removing at most half of the cubes in $\DD_{2^{-(j+1)T}}(X^{(j+1)})$ from each of the sets $Q\cap X^{(j+1)}$ making up $X^{(j,k)}$, we obtain a set $X^{(j)}\subset X^{(j+1)}$ such that $|X^{(j)}|\ge |X^{(j+1)}|/(2T)$ and $|Q\cap X^{(j)}|_{2^{-(j+1)T}}=2^k$ for all $Q\in \DD_{jT}(X^{(j)})$. We see inductively that $|Q\cap X^{(j)}|_{2^{-(j'+1) T}}$ is constant over all $Q\in\DD_{j' T}(X^{(j)})$, for all $j'=j,j+1,\ldots,\ell-1$. The lemma follows by taking $X'=X^{(0)}$.
\end{proof}
We make some remarks on this statement and its proof. Firstly, this is just the simplest example of a flexible and powerful multiscale pigeonholing argument. For example, instead of (or additionally to) uniformizing the branching numbers $N_j$, we can pigeonhole any property of $Q\cap X$, $Q\in\DD_{Tj}(X)$, that depends only on the behavior at scale $2^{-T(j+1)}$ and can be partitioned into a number $C_T$ of classes, with $\log(C_T)/T\to 0$ as $T\to\infty$. Secondly, these ideas can also be used to ``uniformize'' a measure $\mu$ - an additional first step in this case is to pigeonhole a ``$\mu$-large'' $2^{-T\ell}$-set $X$ such that the density of $\mu|_X$ is roughly constant; we can then invoke the argument for sets. Here ``$\mu$-large'' could simply mean that $\mu(X)$ is large, but sometimes it is convenient to look at other quantities like $\|\mu|_X\|_{L^q}$. Lastly, we can iterate such a uniformization lemma to decompose $X$ (or $\mu$) into a union of finitely many ``large'' uniform subsets $X_i$, plus a ``small'' remaining set $X_{\text{bad}}$; see e.g. \cite[Corollary 3.5]{KeletiShmerkin19}.

\subsection{Bourgain's sumset theorem}
\label{subsec:Bourgain-sumset}

Let $X\subset [0,1)$ be a $2^{-m}$-set for some large $m$. We are interested in understanding how the size of the arithmetic sum $X+X = \{ x+y: x,y\in X\}$ relates to the structure of $X$. If $X$ is an interval, then $|X+X|_\delta\sim|X|_\delta$. There are many ``fractal'' sets which satisfy $|X+X|_\delta \le 2^{\e m}|X|_\delta$ with $\e>0$ arbitrarily small: fix a large $T\in\N$, an even larger $\ell\gg T$, and $J\subset \{0,1,\ldots,\ell-1\}$. Let $X_J$ be the set of points in $[0,1)$ whose base $2^{T}$-expansion has a digit zero at position $j+1$ for $j\in J$, but is otherwise arbitrary. Then $X_J+X_J$ has the same structure, except that there could be carries; however, because $T$ is large, these carries will not substantially increase the size of $X_J+X_J$. More precisely,
\[
|X_J+X_J| \le 2^{\ell-|J|}|X_J| \le 2^{\e m}|X_J|, \quad\text{where }\e=1/T,
\]
where as usual we write $m=T\ell$. Note that even though $X_J$ may not look macroscopically like an interval, there is a sequence of scales at which it looks like a union of intervals of the same length, and the left endpoints of these intervals form an arithmetic progression.

The set $X_J$ is $(T;(N_j)_{j=0}^{\ell-1})$-uniform, with $N_j=1$ if $j\in J$, and $N_j=2^T$ otherwise. Bourgain's sumset theorem, which is implicit in \cite{Bourgain10}, and stated in this form in \cite[Corollary 3.10]{Shmerkin19}, asserts that having a small sumset forces this kind of branching structure:
\begin{theorem} \label{thm:Bourgain-sumset}
Given $\delta>0$ there are $\e>0$, $T\in\N$, such that the following holds for all sufficiently large $\ell\in \N$.

Let $m=\ell T$. Suppose $X$ is a $2^{-m}$-set with $|X+X|_{2^{-m}} \le 2^{\e m}|X|_{2^{-m}}$. Then $X$ contains a $(T,(N_j)_{j=0}^{\ell-1})$-uniform subset $X'$ such that:
 \begin{enumerate}[(i)]
 \itemsep0.3em
  \item $|X'|_{2^{-m}}\ge 2^{-\delta m} |X|_{2^{-m}}$.
  \item For each $j$, either $N_j=1$, or $N_j\ge 2^{(1-\delta)T}$.
  \end{enumerate}
\end{theorem}
In other words, up to passing to a large subset, $2^{-m}$-sets with sub-exponential doubling locally look, depending on the scale, like an interval or a point.  This is an example of an ``inverse theorem'' in (discretized) additive combinatorics, in which from a purely combinatorial fact (small doubling) one deduces strong structural information. We will encounter another (related) inverse theorem in \S\ref{subsec:inverse-thm}. We emphasize that Theorem \ref{thm:Bourgain-sumset} does \emph{not} characterize sets with small doubling - even if $X$ is uniform with either full or no branching at each scale, if the locations of the (single) offspring cubes at the scales wit no branching do not have any arithmetic structure, it may well happen that $|X+X|_{2^{-m}}$ is far larger than $|X|_{2^{-m}}$.

\subsection{Bourgain's discretized sum-product and projection theorems}

A heuristic principle of great reach asserts that if $X$ is a subset of some ring, then either the sumset $X+X$ or the product set $X\cdot X$ must be substantially larger than $X$, unless $X$ itself looks like a sub-ring. For example, it is a longstanding conjecture of Erd\H{o}s and Szemer\'{e}di that if $X\subset\Z$, then $\max\{|X+X|,|X\cdot X|\} \gtrsim_{\e} |X|^{2-\e}$ - in other words, either the sumset or the product set must be as large as possible. See \cite{RudnevStevens20} for the best bound at the time of writing, and further discussion.

When dealing with products, it is more convenient to work with subsets of $[1,2)$ rather than $[0,1)$. Again, if $X=[a,b)\subset [1,2)$, then both $|X+X|_\delta$ and $|X\cdot X|_\delta$ are comparable to $|X|_\delta$. Heuristically, one would expect that if $X\subset [1,2)$ does not look roughly like an interval at scales in $[\delta,1]$, then either $|X+X|_{\delta}$ or $|X\cdot X|_{\delta}$ is substantially larger than $|X|_\delta$. This is the content of Bourgain's discretized sum-product theorem, which confirmed a conjecture of Katz and Tao \cite{KatzTao01}:

\begin{theorem}[\cite{Bourgain03, BourgainGamburd08, Bourgain10}] \label{thm:Bourgain-sum-product}
Given $0<\alpha<1$ and $\beta>0$ there are $\kappa(\alpha,\beta)>0$, $\eta=\eta(\alpha,\beta)>0$ such that the following holds for $\delta\le\delta_0(\alpha,\beta)$. Let $X\subset [1,2]$ satisfy $|X|_{\delta}\ge\delta^{-\alpha}$ and
\begin{equation} \label{eq:non-concentration}
|X\cap [t,t+r]|_\delta \le \delta^{-\kappa} r^{\beta}|X|_\delta, \quad t\in [1,2], r\in [\delta,1].
\end{equation}
Then
\[
\max\{|X+X|_{\delta}, |X\cdot X|_{\delta} \} \ge \delta^{-\alpha-\eta}.
\]
\end{theorem}

Hypothesis \eqref{eq:non-concentration} is known as a \emph{non-concentration} assumption, and it quantifies the fact that $X$ ``does not look like an interval''. Note that because of the factor $\delta^{-\kappa}$, it is vacuous at scales close to $1$ or $\delta$. Bourgain \cite{Bourgain03} first proved this theorem under the stronger assumption that \eqref{eq:non-concentration} holds with $\alpha$ in place of $\beta$ (so that the non-concentration exponent matches the size of the set). Bourgain and Gamburd \cite{BourgainGamburd08} then proved it as stated, and used it to establish a spectral gap for subgroups of $\text{SU}(2)$ satisfying a diophantine condition. Under the assumption $\beta=\alpha$, Guth, Katz and Zahl \cite{GKZ21} recently found a simpler proof with an explicit value: any $\eta<\tfrac{\alpha(1-\alpha)}{4(7+3\alpha)}$ works (with $\kappa$ depending also on $\eta$).

In \cite{Bourgain10}, Bourgain proved a discretized projection theorem that can be seen as a far more flexible form of Theorem \ref{thm:Bourgain-sum-product}. Let $\Pi_x(a,b)=a+bx$.
\begin{theorem}[{\cite[Theorem 2]{Bourgain10}}] \label{thm:Bourgain-projection}
Given $0<\alpha<2$ and $\beta>0$ there are $\kappa(\alpha,\beta)>0$, $\eta=\eta(\alpha,\beta)>0$, such that the following holds for $\delta\le\delta_0(\alpha,\beta)$. Let $E\subset [0,1]^2$ satisfy $|E|_\delta \ge \delta^{-\alpha}$ and
\[
|E\cap B(x,r)|_\delta \le \delta^{-\kappa} r^{\beta}|E|_\delta, \quad x\in [0,1]^2, r\in [\delta,1].
\]
Let $X\subset [1,2]$ be a set satisfying \eqref{eq:non-concentration}.

Then there is a set $X_0\subset X$ with $|X\setminus X_0|_\delta \le \delta^\kappa|X|_\delta$, such that if $x\in X_0$ then
\[
|\Pi_x(E')|_\delta \ge \delta^{-\alpha/2-\eta} \quad\text{for all } E'\subset E,\, |E'|_\delta \ge \delta^\kappa |E|_\delta.
\]
\end{theorem}
This is not quite the form the theorem was stated in \cite{Bourgain10} but is formally equivalent; see W. He's article \cite{He20} for this formulation and an extension of Theorem \ref{thm:Bourgain-projection} to projections from $\R^d\to\R^k$. Taking $E=X\times X$ with $|X|_\delta=\delta^{-\gamma}$, we obtain in particular $|X+X\cdot X|_\delta \gtrsim \delta^{-\gamma-\eta}$, which is close to Theorem \ref{thm:Bourgain-sum-product}. One can in fact recover  Theorem \ref{thm:Bourgain-sum-product} from Theorem \ref{thm:Bourgain-projection}, see \cite[Proof of Theorem 1]{Bourgain10}.

The proof of Theorem \ref{thm:Bourgain-projection} relies on Theorem \ref{thm:Bourgain-sumset}. An intermediate step in the proof is showing that if $Y\subset [1,2)$ satisfies the non-concentration assumption \eqref{eq:non-concentration}, then $|Y+x Y|_\delta$ is large for some $x\in X$. If this does not hold, then it is easy to see that $|Y+Y|_\delta$ is also small.  The structural information on $Y$ provided by Theorem \ref{thm:Bourgain-sumset} can then be used (very nontrivially!) to show  that in fact $|Y+ x Y|_\delta$ must be large for some $x\in X$.

Theorem \ref{thm:Bourgain-projection} has striking applications, for example to equidistribution of linear random walks in the torus \cite{BFLM11} and bounds for the dimensions of Kakeya sets in $\R^3$ \cite{KatzZahl19}. We discuss a nonlinear version of the theorem and applications to the Falconer distance set problem in \S\ref{subsec:nonlinear-Bourgain}. For later reference, we conclude this discussion with a Hausdorff dimension version of Theorem \ref{thm:Bourgain-projection}. We note however that it is the discretized version that gets used in the applications.
\begin{theorem}[{\cite[Theorem 4]{Bourgain10}}] \label{thm:Bourgain-projection-dim}
Given $0<\alpha<2$ and $\beta>0$ there is $\eta=\eta(\alpha,\beta)>0$ such that for any Borel set $E\subset \R^2$ with $\hdim(E)\ge\alpha$,
\[
\hdim\{ x\in\R: \hdim(\Pi_x E) < \tfrac{\alpha}{2}+\eta\} \le \beta.
\]
\end{theorem}

\section{Furstenberg's slicing problem}
\label{sec:Furstenberg}

\subsection{Furstenberg's principle and rigidity result}

Recall that $a,b\in \N$ are called \emph{multiplicatively dependent} (denoted $a\sim b$) if $\log a/\log b\in\mathbb{Q}$ or, equivalently, $a$ and $b$ are powers of a common integer. Otherwise, we say that $a$ and $b$ are \emph{multiplicatively independent}, and denote it by $a\nsim b$. If $a\sim b$, say $a=m^{a'}$, $b=m^{b'}$, then there is a straightforward relationship between the expansion of a real number $x$ to bases $a$ and $b$: they are both essentially the expansion to base $m$, looking at it in blocks of $a'$ and $b'$ digits at a time. In the 1960s, H. Furstenberg proposed a series of conjectures which, in different ways, aim to capture the heuristic principle that, on the other hand, \emph{expansions in multiplicatively independent bases have no common structure}.

Recall that if  $a\in\N_{\ge 2}$, we let $T_a:[0,1)\to [0,1)$, $x\mapsto a x\bmod 1$ denote multiplication by $a$ on the circle. A set $X\subset [0,1)$ is $T_a$-invariant if $T_a X\subset X$. Since the map $T_a$ shifts the $a$-ary expansion of a real number, a proper, closed, infinite $T_a$-invariant subset of $[0,1)$ can be thought of as being structured to base $a$. (The full circle $[0,1)$ and finite rational orbits $\{ j/m\}_{j=1}^{m-1}$ are trivially invariant under all $T_a$.) In 1967, Furstenberg \cite{Furstenberg67} proved that no proper infinite closed subset of the circle can be invariant under $T_a$ and $T_b$ if $a\nsim b$. This was the first concrete verification of the above heuristic principle, and gave birth to the vast and ongoing area of \emph{rigidity} in ergodic theory. Furstenberg's $\times 2$, $\times 3$ problem asks whether the natural analog of this result also holds for measures, and is one of the most fundamental open questions in ergodic theory and beyond. He also proposed a number of other conjectures involving $T_a$-invariant sets, that we discuss next.

\subsection{Furstenberg's sumset, slice and orbit conjectures}

In this section $a,b\ge 2$ are multiplicatively independent, and $X,Y\subset [0,1)$ are closed and invariant under $T_a, T_b$. According to Furstenberg's principle, such sets $X, Y$ should have no common structure. Furstenberg's rigidity result established a rough form of this: $X$ and $Y$ cannot be identical, unless trivial. Furstenberg conjectured that $X$ and $Y$ should be not just distinct but ``geometrically independent'', obeying dimensional relationships analogous to those of linear planes in general position. Since these are fractal sets ($T_a$-invariance can be seen as a kind of self-similarity, and it is well known that $\hdim(X)<1$ unless $X=[0,1)$), it is natural to use Hausdorff dimension.

Furstenberg's \emph{sumset conjecture} (which originated in the 1960s but was never stated in print) asserts that
\[
\hdim(X+Y) = \min(\hdim(X)+\hdim(Y),1),
\]
while Furstenberg's \emph{slice} or \emph{intersection conjecture}, stated as Conjecture 1 in \cite{Furstenberg70}, states that
\[
\hdim(X\cap Y) \le \max(\hdim(X)+\hdim(Y)-1,0).
\]
As pointed out in \cite{Furstenberg70}, this latter conjecture easily implies that $X\neq Y$ (unless trivial), recovering the rigidity result. While stopping short of proving the conjecture, Furstenberg in \cite{Furstenberg70} introduced some ideas that are at the heart of modern progress in the area, including what  are now known as \emph{CP-chains}, a class of  Markov chains where the transitions consist in ``zooming in'' dyadically towards typical points for the measures (see  \cite{Furstenberg08} for an elegant formulation of the theory). Using CP-chains, he showed that if $\hdim(X\cap Y)>\gamma$, then for almost all reals $u$ there is a line $\ell_u$ with slope $u$ such that $\hdim((X\times Y)\cap\ell_u)>\gamma$; moreover, there is an ergodic dynamical system on (measures supported on) linear fibers of $X\times Y$ of dimension $>\gamma$.

After partial progress in \cite{PeresShmerkin09}, the sumset conjecture was fully resolved by M. Hochman and the author in \cite{HochmanShmerkin12}, using CP-chains as a key tool. In this work we also introduced the method of \emph{local entropy averages} to bound from below the entropy and dimension of projected images; we will come back to this in \S\ref{subsec:multiscale-entropy}. A simple, purely combinatorial proof was recently obtained by D. Glasscock, J. Moreira and F. Richter \cite{GMR21}.

The slice conjecture was resolved around $10$ years later, independently by the author \cite{Shmerkin19} and my M. Wu \cite{Wu19}. Wu's proof is also based on CP-chains and the ideas from \cite{Furstenberg70}, but introduces a key new ergodic-theoretic insight. A simple conceptual proof, also based on the CP-chains from \cite{Furstenberg70}, was recently obtained by T. Austin \cite{Austin20}. By adapting Wu's method, H. Yu \cite{Yu21} gave a more elementary and quantitative proof in the case $\hdim(X)+\hdim(Y)<1$.  Our proof follows a different approach, based on additive combinatorics and multifractal analysis - we will describe some of the main ideas in the rest of this section. All the proofs yield also Conjecture \ref{conj:Furstenberg}, which was also implicitly stated in \cite{Furstenberg70}. They all also imply the sumset conjecture. Applications of the slice conjecture to number-theoretic problems involving integers with restricted digit expansions were given in \cite{BurrellYu21, GMR21b}.

A further conjecture of Furstenberg \cite[Conjecture 2]{Furstenberg70}, in the authors' view among the hardest and most beautiful in mathematics, asserts that for every irrational $x\in [0,1)$, if $\mathcal{O}_{m,x}=\overline{\{ T_m^n x\}_{n\in\N}}$, is the closure of the orbit of $x$ under $T_m$, then
\[
\hdim(\mathcal{O}_{a,x})+\hdim(\mathcal{O}_{b,x}) \ge 1.
\]
This fits into the theme of lack of common structure for expansions to bases $a,b$: it says that such expansions of an irrational number cannot simultaneously have ``low complexity'', as measured by the dimension of the orbit closure. In particular, if the orbit closure under $T_a$ has ``minimal complexity'' (dimension $0$), then the $T_b$-orbit must be dense, meaning that every possible $b$-ary block appears in the base $b$ expansion of $x$. This conjecture is wide open; even proving that either $\hdim(\mathcal{O}_{a,x})$ or $\hdim(\mathcal{O}_{b,x})$ has positive dimension  seems to require completely new ideas. However, it is a formal consequence of the slicing conjecture that the set of $x$ for which the orbit conjecture fails has Hausdorff dimension zero. Unfortunately, this says nothing about points $x$ for which $\hdim(\mathcal{O}_{a,x})=0$, since all such points form a zero dimensional set. Recently, B. Adamczewski and C. Faverjon \cite{AdamczewskiFaverjon20} showed that an irrational number cannot be automatic in bases $a$ and $b$; being automatic is a computational notion of ``simplicity'', and so this can be seen as a first verification that an irrational number cannot be ``too simple'' in two multiplicatively independent bases.

\subsection{$L^q$ dimensions, self-similarity, and the dimension of slices}
\label{subsec:Lq-selfsim}

Let $\mathcal{P}(X)$ denote the family of Borel probability measures on a metric space $X$. Given $\mu\in\mathcal{P}(\R^d)$, the $L^q$ dimensions $\{D_\mu(q)\}_{q>1}$ are a family of indices measuring the degree of singularity of $\mu$ through its $q$-moments:
\[
D_\mu(q) = D(\mu,q) = \liminf_{\delta\to 0} \frac{\log\sum_{Q\in\DD_{\delta}} \mu(Q)^q}{(q-1)\log\delta}.
\]
(It is also possible to define $D_\mu(q)$ for $q<1$, but we do not need this here.) The normalizing factor $1/(q-1)$ ensures that $D_\mu(q)\in [0,d]$. If $\mu$ has an $L^q$ density, then $D_\mu(q)=1$ but $D_\mu(q)<1$ is possible even for other absolutely continuous measures. For any fixed $\mu$, the function $D_\mu$ is non-increasing, so it makes sense to define
\[
D_\mu(\infty) = D(\mu,\infty)= \lim_{q\to\infty} D_\mu(q).
\]
It is not hard to show that $D_\mu(\infty)$ is the supremum of the $s$ such that $\mu(B(x,r)) \le C r^s$ for some constant $C=C(\mu,s)$ and all closed balls $B(x,r)$. Such $s$ are also called \emph{Frostman exponents} of $\mu$. The function $\tau_\mu(q)=(q-1)D_\mu(q)$ is known as the \emph{$L^q$-spectrum} of $\mu$.  It is always concave. In particular, both $\tau_\mu$ and $D_\mu$ are differentiable outside of a countable set of $q$.  See \cite[Section 3]{LauNgai99} for proofs of these facts and further background on the $L^q$ spectrum and dimension.

We are interested in upper bounds for the dimension of slices. The next very simple but key lemma relates this problem to lower bounds on $D_\mu(\infty)$ for suitable measures $\mu$. Given a map $\pi:X\to Y$ and $\mu\in\PP(X)$, we denote the push-forward measure by $\pi\mu=\mu\circ\pi^{-1}$.
\begin{lemma} \label{lem:Frostman-exp-to-small-fiber}
Suppose $\pi:\R^d\to \R$ is a Lipschitz map. Let $\mu\in\PP([0,1]^d)$ be such that $\mu(B(x,r)) \ge c r^\alpha$ for all $x\in X:=\supp(\mu)$, $r\in (0,1]$. If $D(\pi\mu,\infty)\ge \beta$, then
\[
\hdim(X\cap \pi^{-1}(y)) \le \alpha-\beta \quad\text{for all }y\in\R.
\]
\end{lemma}
\begin{proof}
Fix a small $\e>0$ and $y\in\R$. Let $(x_j)_{j=1}^M$ be a maximal $\e$-separated subset of $X\cap \pi^{-1}(y)$,  and let $A=\bigcup_{j=1}^M B(x_j,\e/2)$. Since the balls are disjoint, $\mu(A)\ge c M (\e/2)^\alpha$. On the other hand, $\pi A$ is contained in an interval of size $\lesssim \e$ and hence, for any $\eta>0$,
\[
\mu(A)\le (\pi\mu)(\pi A) \lesssim_\eta \e^{\beta-\eta}.
\]
Comparing the bounds, $M\lesssim_{c,\eta} \e^{\beta-\alpha-\eta}$. Now $X\cap \pi^{-1}(y)\subset \bigcup_{j=1}^M B(x_j,\e)$ by the maximality of $(x_j)$, and hence $X\cap \pi^{-1}(y)$ can be covered by $\lesssim_{c,\eta} \e^{\beta-\alpha-\eta}$ balls of radius $\e$. Letting $\eta\to 0$, we get the claim.
\end{proof}

In order to connect this lemma to the slice conjecture, our next step is to look at measures defined on invariant sets. A set $X\subset \R$ is \emph{self-similar} if there are finitely many contracting similarity transformations $f_i(x)=r_i x+ t_i$, $i\in I$ with $0<r_i<1$, such that $X=\cup_{i\in I} f_i(X)$. The family $(f_i)_{i\in I}$ is called an \emph{iterated function system (IFS)} and $X$ is its \emph{attractor}. For simplicity, from now we assume that we are in the homogeneous case, meaning that all the contractions $r_i$ are equal. 

A closed $T_a$-invariant set $X$ needs not be self-similar in the sense above. However, it is easy to see \cite[p.378]{Shmerkin19} that for every $\e>0$ there is a set $X'\supset X$ with $\hdim(X')<\hdim(X)+\e$, which is the attractor of an IFS of the form $\{ a^{-m}(x+j)\}_{j\in J}$, where $m$ and the ``digit set'' $J\subset \{0,\ldots,a^m-1\}$ depend on $\e$. Hence, in order to establish Conjecture \ref{conj:Furstenberg}, we may assume that $X, Y$ are self-similar of this special form. Since the assumption $a\nsim b$ is not affected by taking powers, we assume that $m=1$ for simplicity.

Given a homogeneous IFS $\mathcal{I}=\{ r x+ t_i\}_{i\in I}$, let $\Delta=\Delta_{\mathcal{I}}=\tfrac{1}{|I|}\sum_{i\in I}\delta_{t_i}$, where $\delta_t$ denotes a unit mass at $t$, and define the (natural) \emph{self-similar measure}
\[
\mu= \mu_{\mathcal{I}} = \ast_{n=0}^\infty S_{r^n}\Delta,
\]
where $S_u=u x$ scales by $u$. In other words, $\mu$ is the push-forward of $\prod_{n=0}^\infty \Delta$ under $(x_n)_{n=0}^\infty\mapsto \sum_{n=0}^\infty x_n r^n$. Then $\mu$ is supported on the attractor $X$, and it easy to check that for $\alpha=\log|I|/\log(1/r)$,
\[
\mu(B(x,r)) \ge c r^\alpha, \quad x\in X, r\in (0,1].
\]
The parameter $\alpha$ is the \emph{similarity dimension} of the IFS $\mathcal{I}$; if the pieces $(f_i(X))_{i\in I}$ are disjoint, then it equals $\hdim(X)$, but it is a well known open problem to understand when equality holds in the overlapping situation; see \cite{Hochman14} and P. Varj\'{u}'s survey in this volume for progress on this problem.

Fix closed $T_a, T_b$-invariant self-similar sets $X, Y$ as above, and let $\mu_X,\mu_Y$ be the corresponding self-similar measures, defined in terms of atomic measures $\Delta_X, \Delta_Y$. Let $\alpha=\hdim(X)$, $\beta=\hdim(Y)$.
As we have seen,
\[
(\mu_X\times\mu_Y)(B(p,r)) \ge c r^{\alpha+\beta}, \quad p\in X\times Y, r\in (0,1].
\]
Recall that $\Pi_u(x,y)=x+u y$. Then $\Pi_u(\mu_X\times\mu_Y) = \mu_X \ast S_u\mu_Y$. By the above discussion and Lemma \ref{lem:Frostman-exp-to-small-fiber}, in order to prove Conjecture \ref{conj:Furstenberg}, it is enough to show:
\begin{theorem} \label{thm:Frostman-convolution}
\begin{equation} \label{eq:Frostman-convolution}
D(\mu_X\ast S_u\mu_Y,\infty)= \min(\alpha+\beta,1) \quad\text{for all }u\neq 0.
\end{equation}
\end{theorem}
This recasts the slice conjecture into a problem concerning projections and self-similarity. This is convenient, since a lot was previously known about this topic. For example, \eqref{eq:Frostman-convolution} was known to hold for Hausdorff dimension in place of $L^\infty$ dimension \cite{HochmanShmerkin12} and even for $L^q$ dimension for $q\in (1,2]$ \cite{NPS12}. However, these results used in an essential way the known fact that for arbitrary measures $\mu,\nu$, Eq. \eqref{eq:Frostman-convolution} with $q\in (1,2]$ in place of $\infty$ holds for almost every $u$. This is not true for $q>2$ and hence new ideas were needed. While the setting is different, the inspiration came from M. Hochman's work on self-similarity, see the survey \cite{Hochman18} for an overview.

\subsection{Dynamical self-similarity and exponential separation}

While $\mu_X$ and $\mu_Y$ are self-similar measures in the sense described in \S\ref{subsec:Lq-selfsim}, the convolution $\mu_X \ast S_u \mu_Y$ is not strictly self-similar  since $a\nsim b$. However, it satisfies a more flexible notion that we term \emph{dynamical self-similarity}. Suppose $a<b$, and let us define $\mathbb{G}=[0,\log b)$,  $\mathbf{T}:\mathbb{G}\to \mathbb{G}$, $x\mapsto x+\log a\bmod(\log b)$. For each $x\in \mathbb{G}$, let
\begin{equation} \label{eq:def-Delta}
\Delta(x) = \left\{
\begin{array}{ll}
  \Delta_X\ast S_{e^x}\Delta_Y & \text{if } x\in [0,\log a) \\
  \Delta_X & \text{if } x\in [\log a,\log b)
\end{array}
\right..
\end{equation}
These are finitely supported measures.  It is easy to check (see \cite[\S 1.4]{Shmerkin19}) that
\[
\nu_x := \mu_X \ast S_{e^x} \mu_Y = \ast_{n=0}^\infty S_{a^{-n}}(\Delta(\mathbf{T}^n x)).
\]
This is what we mean by dynamical self-similarity: $\nu_x$ has a structure analogous to that of $\mu_X, \mu_Y$, but the discrete measure $\Delta$ now depends on the scale and is driven by the dynamics of $\mathbf{T}$. Note that
\begin{equation} \label{eq:dynamic-ssm}
\nu_x = \nu_{x,n} \ast S_{a^{-n}}\nu_{\mathbf{T}^n x},\quad\text{where } \nu_{x,n} = \ast_{j=0}^{n-1} \Delta(\mathbf{T}^j x).
\end{equation}
This says that $\nu_x$ is a convex combination of scaled down copies, not quite of itself (as in the strictly self-similar case), but of the related measures $\nu_{\mathbf{T}^n x}$.

In the proof of Theorem \ref{thm:Frostman-convolution}, dynamical self-similarity plays a central r\^{o}le. Another key feature is \emph{exponential separation}. The measures $\nu_{x,n}$ defined in \eqref{eq:dynamic-ssm} are purely atomic; let $\mathcal{A}_x(n)$ denote the set of its atoms. Then
\[
|\mathcal{A}_x(n)| \le \prod_{j=0}^{n-1}|\supp(\Delta(\mathbf{T}^j x))|.
\]
Let $M_x(n)$ denote the minimal separation between two elements of $\mathcal{A}_x(n)$, defined to be $0$ if the inequality above is strict. We claim that there is a number $c>0$ such that
\begin{equation} \label{eq:exp-sep}
M_x(n) \ge c^n \quad\text{for } n\ge n_0(x), \text{ for Lebesgue almost all } x\in \mathbb{G}.
\end{equation}
Indeed, the distance between two elements of $\mathcal{A}_x(n)$ has the form $i a^{-n} + e^x j b^{-n}$ for some $|i|<a^n$, $|j|<b^n$, and $i,j$ are not both $0$. If $j=0$, then $i\neq 0$ and the distance is $\ge a^{-n}$. Otherwise, $x\mapsto i a^{-n} + e^x j b^{-n}$ has derivative $\ge b^{-n}$ in absolute value, and so is $\ge c^{n}$ in absolute value outside of a set of $x$ of measure $2 (cb)^n$. Since there are $\lesssim (ab)^n$ pairs $i,j$, we see that $M_x(n)\ge c^n$ outside of a set of measure $\lesssim (c\cdot ab^2)^n$. Hence if $c< (ab^2)^{-1}$, then Borel-Cantelli yields \eqref{eq:exp-sep}.

Exponential separation was introduced in the self-similar setting by Hochman \cite{Hochman14}. The way we apply it will be conceptually similar. However, in the strictly self-similar setting, this condition is often hard to check (or fails) for concrete examples, while as we have seen, in the dynamical setting the one-dimensional group $\mathbb{G}$ makes the verification straightforward.

A final ingredient of the proof of Theorem \ref{thm:Frostman-convolution} is \emph{unique ergodicity}: the dynamical system $(\mathbb{G},\mathbf{T})$ is isomorphic to a $(\log a/\log b)$-rotation on the circle. Because $a\nsim b$, Lebesgue measure on the circle is the only $\mathbf{T}$-invariant measure on $\mathbb{G}$: this is the point in the proof where the hypothesis $a\nsim b$ gets used. As we will see, this will be crucial in obtaining information for \emph{every} $x\in\mathbb{G}$ out of seemingly weaker information for \emph{almost every} $x\in\mathbb{G}$.

In the rest of this section, we indicate how dynamical self-similarity, exponential separation and unique ergodicity enter into the proof of Theorem \ref{thm:Frostman-convolution}. The theorem extends to a more general setting in which appropriate versions of these three properties hold (plus some additional technical assumptions): see \cite[\S 1.5]{Shmerkin19}.

\subsection{A subadditive cocycle and the r\^{o}le of unique ergodicity}

Fix $q\in (1,\infty)$, recall that $\nu_x = \mu_X * S_{e^x}\mu_Y = \Pi_{e^x}(\mu_X\times \mu_Y)$, and let
\[
\phi_{q,n}(x) = \log\left(\sum_{I\in\DD_{2^{-n}}} \nu_x(I)^q\right),\quad x\in\mathbb{G},
\]
where here and below logarithms are to base $2$. In order to establish Theorem \ref{thm:Frostman-convolution}, it is enough to show that
\begin{equation} \label{eq:llimit-Lq-dim}
\liminf_{n\to\infty} \frac{\phi_{q,n}(x)}{-(q-1)n} \ge \min(\alpha+\beta,1), \quad \text{for all } x\in\mathbb{G}.
\end{equation}
Indeed, it is rather easy to check that for any $x\in\mathbb{G}$
\[
\limsup_{n\to\infty} \frac{\phi_{q,n}(x)}{-(q-1)n} \le \min(\alpha+\beta,1),
\]
and so \eqref{eq:llimit-Lq-dim} yields $D(\mu_X \ast S_u \mu_Y,q) =\max(\alpha+\beta,1)$ (and the limit in the definition of $L^q$ dimension exists), from where the claim follows by taking $q\to\infty$. A priori this is only true for $u=e^x\in [1,b)$, but using self-similarity it is not hard to extend it to every $u\neq 0$.

Dynamical self-similarity and the convexity of $t^q$ imply (see \cite[Prop. 4.6]{Shmerkin19})
\[
\phi_{q,n+m}(x) \le C_q + \phi_{q,n}(x) + \phi_{q,m}(\mathbf{T}^n x).
\]
Hence $(\phi_{q,n}+C_q)_n$ is a \emph{subadditive cocycle} over the dynamical system $(\mathbb{G},\mathbf{T})$. The functions $\phi_{q,n}$ are continuous on $\mathbb{G}$ except at $x=\log a$. The unique ergodicity of $(\mathbb{G},\mathbf{T})$ can then be seen to imply (\cite[\S 4.2]{Shmerkin19}) that there is a number $D(q)$ such that
\begin{align*}
\liminf_{n\to\infty} \frac{\phi_{q,n}(x)}{-(q-1)n} &= D(q)\quad\text{for \emph{all} } x\in \mathbb{G},\\
\lim_{n\to\infty}  \frac{\phi_{q,n}(x)}{-(q-1)n} &= D(q) \quad\text{for almost all } x\in\mathbb{G}.
\end{align*}
Hence the task is now to show that $D(q)=\max(\alpha+\beta,1)$. This is a really crucial point, because one only needs to compute the almost sure limit $D(q)$ in order to reach a conclusion valid for \emph{every} $x$. This is also the strategy from \cite{NPS12} in the case $q\le 2$; the almost sure statement follows in that case by classical projection results, while the  more involved argument discussed below is required when $q>2$. Because the $L^q$ dimension is continuously decreasing in $q$, it is easy to check that $D=D_{\nu_x}$ (as a function) for almost all $x$; in particular, $D$ is differentiable outside of a countable set.

\subsection{An inverse theorem for the $L^q$ norms of convolutions}
\label{subsec:inverse-thm}

So far, discretized additive combinatorics has not entered the picture. As indicated earlier, the proof of Theorem \ref{thm:Frostman-convolution} is inspired by Hochman's work on self-similar sets and measures \cite{Hochman14}. Hochman \cite[Theorem 2.7]{Hochman14} proved an inverse theorem for the entropy of convolutions of general measures on $\R$, then applied it to self-similar measures, and concluded that under exponential separation they have the ``expected'' dimension; again we refer to \cite{Hochman18} for a survey of these ideas. We follow a parallel strategy; in particular, we rely on a new inverse theorem for the $L^q$ norms of convolutions.

If $\nu$ is finitely supported, we denote $\|\nu\|_q^q= \sum_x \nu(x)^q$ for $q\in (1,\infty)$. If $\mu,\nu$ are supported on $2^{-m}\Z\cap [0,1)$ then, by Young's inequality (which in this setting is just the convexity of $t^q$),
\begin{equation} \label{eq:Young}
\|\mu\ast\nu\|_q \le \|\mu\|_q\|\nu\|_1.
\end{equation}
We are interested in understanding what happens when we are close to equality,  in an exponential sense (up to $2^{-\e m}$ factors). This is the case if $\mu$ is the uniform measure on $2^{-m}\Z\cap [0,1)$, or if $\nu$ is supported on a single atom, but also in some ``fractal'' situations. For example, if $\mu=\nu$ is the uniform measure on the (left endpoints of the intervals making up the) sets $X_J$ from \S\ref{subsec:Bourgain-sumset}; it is also possible to construct similar examples with $\mu$ different from $\nu$. Our inverse theorem asserts, roughly speaking, that if we are close to equality in \eqref{eq:Young}, then \emph{locally} either $\mu$ looks very uniform or $\nu$ looks like an atom.

\begin{theorem}[{\cite[Theorem 2.1]{Shmerkin19}}]
\label{thm:inverse-thm}
For each $q>1$, $\delta>0$, there are $T\in\N$, $\e > 0$, such that the following holds for $\ell\ge \ell_0(q,\delta)$. Let $m=\ell T$ and let $\mu, \nu\in\PP(2^{-m}\Z\cap [0,1))$. Suppose
\[
 \|\mu \ast \nu\|_q \geq 2^{-\e m} \|\mu\|_q.
\]
Then there exist sets $X\subset \supp\mu$ and $Y\subset\supp\nu$, so that
\begin{enumerate}[(i)]
\itemsep0.3em
 \item $\| \mu|_X \|_q \geq 2^{-\delta m}\|\mu\|_q$ and $\| \nu|_Y \|_1 = \nu(Y) \geq 2^{-\delta m}$.
 \item $\mu(x_1) \leq 2 \mu(x_2)$ for all $x_1,x_2\in X$; and $\nu(y_1)\leq 2\nu(y_2)$ for all $y_1,y_2\in Y$.
 \item $X$ and $Y$ are $(T;\ell)$-uniform; let $(N_j)_{j=0}^{\ell-1}$, $(N'_j)_{j=0}^{\ell-1}$ be the associated sequences.
 \item \label{it:inv-thm:iv} For each $0\le i<\ell$, either $N_j \ge 2^{(1-\delta)T}$ or $N'_j=1$ (or both).
 \end{enumerate}
\end{theorem}
The reader will note the analogy with Theorem \ref{thm:Bourgain-sumset}, especially in the case $\mu=\nu$. In fact, Theorem \ref{thm:Bourgain-sumset} is a central component of the proof of Theorem \ref{thm:inverse-thm}. In order to pass from the size of sumsets to the $L^q$ norm of convolutions, we use the celebrated Balog-Szemer\'{e}di-Gowers (BSG) Theorem, see \cite[\S 2.5]{TaoVu06}. Simplifying slightly, the BSG Theorem asserts that if $\|\mu*\mu\|_2 \ge K^{-1} \|\mu\|_2$ for  $\mu\in\PP(\Z)$, then there is a set $X$ such that $\mu(X)\ge K^{-C}$ and $|X+X|\le K^C |X|$, where $C>0$ is universal. To be more precise, this holds if $\mu$ is the uniform measure on some set $X_0$. In the case $\mu=\nu$ and $q=2$, the claim is little more than the BSG Theorem combined with Theorem \ref{thm:Bourgain-sumset} and some dyadic pigeonholing. To deal with the general case, we appeal to an asymmetric version of BSG, \cite[Theorem 2.35]{TaoVu06}, while the general case $q\in (1,\infty)$ can be reduced to the case $q=2$ by an application of H\"{o}lder's inequality \cite[Lemma 3.4]{Shmerkin19}. We remark that the theorem fails at $q=1$ and $q=\infty$ due to lack of strict convexity; this is the reason why, even though we are ultimately interested in $L^\infty$ dimensions, we work with $L^q$ dimensions throughout the proof.

While motivated by the slice conjecture, Theorem \ref{thm:inverse-thm} is a result in geometric measure theory. In \cite{RossiShmerkin20}, E. Rossi and the author applied it to the growth of $L^q$ dimension under convolution. It also features in two recent results of T. Orponen \cite{Orponen21, Orponen21b} concerning projections of planar sets outside of a zero-dimensional set of directions.

\subsection{Conclusion of the proof: sketch}

We indicate very briefly how the proof of Theorem \ref{thm:Frostman-convolution} (and hence of Conjecture \ref{conj:Furstenberg}) is concluded. Given a measure $\mu$ on $\R$ we let $\mu^{(m)}$ be the purely atomic measure with
\[
\mu^{(m)}(j 2^{-m})=\mu([j2^{-m},(j+1)2^{-m})).
\]
Thus, $\mu^{(m)}$ is a discrete approximation to $\mu$ at scale $2^{-m}$. Note that $\phi_{q,m}(x)=\log \|\nu_x^{(m)}\|_q^q$. The inverse theorem is used to show:
\begin{theorem}[{\cite[Theorem 5.1]{Shmerkin19}}] \label{thm:dyn-ssm-smoothening}
Fix $q\in (1,\infty)$ such that $D$ is differentiable at $q$ and $D(q)<1$. For every $\sigma>0$ there is $\e=\e(\sigma,q)>0$ such that if $m\ge m_0(\sigma,q)$,  and $\rho\in\PP(2^{-m}\Z\cap [0,1))$ satisfies $\|\rho\|_q \le 2^{-\sigma m}$, then
\[
\| \nu_x^{(m)} \ast \rho\|_q \le 2^{-(D(q)+\e)m},\quad x\in \mathbb{G}.
\]
\end{theorem}
The assumption  $\|\rho\|_q \le 2^{-\sigma m}$ says that $\rho$ is not too close to being atomic in the $L^q$ sense. Since $D(q)=D_{\nu_x}(q)$ for almost all $x$, the theorem says that convolving with any quantitatively non-atomic measure results in a smoothening of the $L^q$ norm of $\nu_x$ at small scales (unless $D(q)=1$, in which case $\nu_x$ was already ``maximally smooth''). This is, again, a dynamical, $L^q$ version of a result of Hochman, \cite[Corollary 5.5]{Hochman14}. Heuristically, this is deduced from Theorem \ref{thm:inverse-thm} as follows: assuming the conclusion fails, let $X, Y$ be the sets provided by the inverse theorem. Using that $\|\rho\|_q \le 2^{-\sigma m}$, one can see that $Y$ has positive branching ($N'_j>1$) for a positive proportion of scales $j$. Then by \ref{it:inv-thm:iv}, $X$ must have ``almost full branching'' $(N_j \ge 2^{(1-\delta)T})$ at those scales. But the dynamical self-similarity of $\nu_x$ can be used to rule this out, since it implies that $\nu_x$ should have ``roughly constant branching'', which is less than full since $D(q)<1$. Making this precise is one of the biggest hurdles in the proof of Theorem \ref{thm:Frostman-convolution}; it relies on ideas from multifractal analysis, in particular, the fact that if $D'(q)$ exists then $\|\mu_x^{(m)}\|_q$ is heavily concentrated on points of mass $\approx 2^{T'(q)m}$, where $T=(q-1)D$.

Once Theorem \ref{thm:dyn-ssm-smoothening} is in hand, the rest of the proof of Theorem \ref{thm:Frostman-convolution} is a fairly straightforward adaptation of Hochman's arguments. Theorem \ref{thm:dyn-ssm-smoothening} is used to show that (always assuming $D'(q)$ exists and $D(q)<1$)
\[
\lim_{n\to\infty}\frac{\log \|\nu_{x,n}^{(R n)}\|_q^q}{(q-1)n\log(1/a)} = D(q)\quad\text{for any  } R>\log a \text{ and almost all }x\in\mathbb{G},
\]
where $\nu_{x,n}$ is the discrete approximation to $\nu_x$ defined in \eqref{eq:dynamic-ssm}. See \cite[Proposition 5.2]{Shmerkin19}. Now the exponential separation \eqref{eq:exp-sep} comes into play: if $R$ is taken large enough in terms of $c$, then the atoms of $\nu_{x,n}$ are $2^{-R n}$-separated for $n\ge n_0(x)$, and this easily yields
\[
\log\|\nu_{x,n}^{(R n)}\|_q^q = \log\|\nu_{x,n}\|_q^q = (1-q)\sum_{j=0}^{n-1} \log|\supp(\Delta(\mathbf{T}^j x))|.
\]
Recalling \eqref{eq:def-Delta}, the ergodic theorem can then be used to conclude that if $D(q)<1$, then $D(q)=\alpha+\beta$, completing the proof.

\subsection{Extensions and open problems}

\subsubsection{Slices of McMullen carpets}
The set $X\times Y$ in Conjecture \ref{conj:Furstenberg} is invariant under the toral endomorphism $T_a\times T_b$, but there are many closed invariant sets under $T_a\times T_b$ which are not cartesian products. The simplest class are McMullen carpets: given $J\subset \{0,\ldots,a-1\}\times \{0,\ldots,b-1\}$, let
\[
E_J = \left\{ \big(\sum_{n=1}^\infty x_n a^{-n}, \sum_{n=1}^\infty y_n b^{-n}\big): (x_n,y_n)\in J \text{ for all } n \right\}.
\]
If $J=J_1\times J_2$ then we are in the setting of Conjecture \ref{conj:Furstenberg}, but otherwise the methods of \cite{Shmerkin19, Wu19} do not directly apply. One new difficulty is that these carpets often have different Hausdorff, Minkowski and Assouad dimension, while these all coincide in the product case. Nevertheless, by modifying the method of Wu, A. Algom \cite{Algom20} proved an upper bound for the dimension of linear slices of McMullen carpets, that reduces to Conjecture \ref{conj:Furstenberg} in the product case. The bound was recently improved further by A. Algom and M. Wu \cite{AlgomWu21}, but the optimal result remains elusive.

\subsubsection{Bernoulli convolutions} Given $\lambda\in (1/2,1)$, we define the \emph{Bernoulli convolution} (BC) $\nu_\lambda=\ast_{n=0}^\infty S_{\lambda^n}\Delta$, where $\Delta=\tfrac{\delta_{-1}+\delta_1}{2}$. This is the simplest family of overlapping self-similar measures, yet it remains a major open problem with deep connections to number theory to elucidate their properties. BCs are extensively discussed in \cite{Hochman18} and in P. Varj\'{u}'s article in this volume, so here we only point out that the method of proof discussed in this section also yields that $D(\nu_\lambda,\infty)=1$ for all $\lambda$ with exponential separation (a set of Hausdorff co-dimension zero) and $\nu_\lambda$ is absolutely continuous with a density in $L^q$ for all $q\in (1,\infty)$, for all $\lambda$ outside of a (non-explicit) set of exceptions of zero Hausdorff dimension. See \cite[Section 9]{Shmerkin19}. In a major breakthrough, P. Varj\'{u} \cite{Varju19} proved that $\nu_\lambda$ has Hausdorff dimension $1$ (which is weaker than $D(\nu_\lambda,q)=1$ if $q>1$) for \emph{all} transcendental $\lambda$. It remains a challenge to extend Varj\'{u}'s result to $L^\infty$ and even to $L^q$ dimensions.

\subsubsection{Higher dimensions} A natural higher dimensional version of Conjecture \ref{conj:Furstenberg} involves slicing the product of closed sets $(X_i)_{i=1}^d$ invariant under $(T_{a_i})_{i=1}^d$, with affine subspaces. As another application of the dynamical self-similarity framework, we have:
\begin{theorem} \label{thm:slice-hyperplanes}
Let $X_i$ be closed, $T_{a_i}$-invariant sets, $i=1,\ldots, d$, with $a_i\nsim a_j$ for $i\neq j$. Then
\[
\hdim( (X_1\times\cdots \times X_d)\cap H)\le \max(\hdim(X_1)+\cdots+\hdim(X_d)-1,0)
\]
for all affine hyperplanes $H\subset\R^d$ not containing a line in a coordinate direction.
\end{theorem}
The case $d=2$ is Conjecture \ref{conj:Furstenberg}. The higher dimensional case follows in a similar way, using \cite[Theorem 1.11]{Shmerkin19} and Lemma \ref{lem:Frostman-exp-to-small-fiber} for projections from $\R^d$ to $\R$, although verifying the exponential separation assumption takes a little bit of work, see \cite{Shmerkin21}. We underline that it seems hard to prove such a result using the approaches of \cite{Wu19, Austin20}. To be more precise, it is possible but under the more restrictive assumption that $(1/\log a_i)_{i=1}^d$ are linearly independent over $\mathbb{Q}$. This is unknown in most cases, for example for $2,3,5$.

What about slicing with lower dimension subspaces? For this, we need to consider projections $\Pi:\R^d\to\R^k$ and in turn this requires an inverse theorem for convolutions in $\R^k$. This is necessarily more challenging because there is a new obstruction to smoothening of convolutions: having the measures (locally) concentrated on lower dimensional subspaces. Nevertheless, Hochman \cite{Hochman17} proved an inverse theorem for the entropy of convolutions in arbitrary dimension. In \cite{Shmerkin21}, using Hochman's result, we derive an $L^q$ version, and use it to deduce a generalization of Theorem \ref{thm:slice-hyperplanes} to slices with planes of arbitrary dimension.

\section{Falconer's distance set problem}

\label{sec:Falconer}

\subsection{Introduction}

We now discuss Conjecture \ref{conj:Falconer}. It is a natural continuous analog of the P. Erd\H{o}s distinct distances conjecture, stating that $N$ points in $\R^d$ determine $\gtrsim_{d,\e} N^{2/d-\e}$ distinct distances. Erd\H{o}s' conjecture was famously resolved in the plane by L. Guth and N. Katz \cite{GuthKatz15}, but the techniques they used seem hard to extend to the continuous setting. As shown already by Falconer \cite{Falconer85} , the measurability condition in Conjecture \ref{conj:Falconer} is crucial.

From now on fix a Borel set $X\subset\R^d$. Falconer \cite{Falconer85} proved that $|\Delta(X)|>0$ provided $\hdim(X)>(d+1)/2$  (here and below, $|\cdot|$ denotes Lebesgue measure as well as cardinality). In the plane, the threshold $3/2$ was lowered successively to $13/9$ by J. Bourgain \cite{Bourgain94}, to $4/3$ by T. Wolff \cite{Wolff99}, and recently to $5/4$ by L. Guth, A. Iosevich, Y. Ou and H. Wang \cite{GIOW20}. There have been parallel developments in higher dimensions \cite{Erdogan05, DGOWWZ21, DuZhang19, DIOWZ21}.  These results use deep methods from restriction theory in harmonic analysis; the connection to restriction was made by P. Mattila \cite{Mattila87}, through what has become known as the \emph{Mattila integral}. B. Liu \cite{Liu19} found a pinned version of the Mattila integral; that is, with $\Delta(X)$ replaced by $\Delta_y(X)=\{|x-y|:x\in X\}$. As a result, all the previous results are also valid for pinned distance sets. Summarizing, the current world records are \cite{DuZhang19,GIOW20,DGOWWZ21,DIOWZ21}: let
\[
\alpha_d = \left\{
\begin{array}{lll}
  \frac{d}{2}+\frac{1}{4} & \text { if } & \text{$d$ is even} \\
  \frac{d}{2}+\frac{1}{4}+\frac{1}{8d-4} & \text{ if } & \text{$d$ is odd}
\end{array}
\right..
\]
Then for a Borel set $X\subset\R^d$ with $\hdim(X)>\alpha_d$ there is $y\in X$ such that $|\Delta_y(X)|>0$.

What if we assume $\hdim(X)=d/2$ instead? Falconer \cite{Falconer85} proved that in this case $\hdim(\Delta(X))\ge 1/2$. There are at least three reasons why this is a natural barrier to overcome. (i) If $R$ was a $1/2$-dimensional Borel subring of the reals, then the distance set of $X=R\times\cdots\times R\subset \R^d$ would be contained in a locally Lipschitz image of $R$. By the product formula for dimension, $\hdim(X)\ge d/2$, so if $R$ existed then Falconer's bound would be sharp. As it turns out, no such Borel subring exists \cite{EdgarMiller03}, but this was an open problem for nearly 40 years. (ii) For a natural single-scale version of the problem, the exponent $1/2$ is actually sharp. This is the ``train track'' example introduced by N. Katz and T. Tao \cite{KatzTao01}: given a small scale $\delta>0$, let $X\subset [0,1]^2$ be the union of $\sim \delta^{-1/2}$ equally spaced vertical rectangles of size $\delta\times\delta^{1/2}$, with a $\delta^{1/2}$ space between consecutive rectangles. See \cite[Figure 1]{KatzTao01}. Then $|X|_\delta \sim \delta^{-1}$ and
\[
 |X\cap B(x,r)|_{\delta} \sim r |X|_\delta, \quad x\in X, r\in [\delta,1].
\]
Hence $X$ looks very much like a set of dimension $1$ (even Ahlfors regular) down to resolution $\delta$. Yet, the set of distances between two separated rectangles is contained in an interval of length $\lesssim\delta$, and this can be used to show that $|\Delta(X)|_\delta \sim \delta^{-1/2}$. (iii) Finally, if the Euclidean norm is replaced by the $\ell_\infty$ norm, then again it is not hard to see that the threshold $1/2$ is sharp, so any improvement must exploit the curvature of the Euclidean norm. We also emphasize that even though the harmonic analytic methods described above also yield dimension estimates when $\hdim(X)\le \alpha_d$, they do not say anything for $\hdim(X)=d/2$.

Despite these challenges, we have:
\begin{theorem}[Katz-Tao \cite{KatzTao01}, Bourgain \cite{Bourgain03}]  \label{thm:Bourgain-distance}
There is a universal $\eta>0$ such that if $X\subset\R^2$ is a Borel set with $\hdim(X)\ge 1$, then $\hdim(\Delta(X))\ge 1/2+\eta$.
\end{theorem}
Katz and Tao \cite{KatzTao01}  proved that the discretized sum-product conjecture (Theorem \ref{thm:Bourgain-sum-product}) implies the above theorem. As we saw, Bourgain \cite{Bourgain03} then proved Theorem \ref{thm:Bourgain-sum-product}. In order to avoid ``train track'' examples, Katz and Tao had as an intermediate step a ``discretized bilinear'' version of Falconer's problem. This approach does not seem to extend to pinned distance sets. The value of $\eta$, although effective in principle, is hard to track down and would in any event be tiny (recall that the conjecture is $\eta=1/2$).

\subsection{A nonlinear version of Bourgain's projection theorem}
\label{subsec:nonlinear-Bourgain}

There is a formal analogy between Theorems \ref{thm:Bourgain-projection-dim} and \ref{thm:Bourgain-distance}: both provide an ``$\eta$-impovement'' over a natural barrier, and as we saw they are both connected to discretized sum-product. We take this analogy further. We can view $\{ \Delta_y(x) =|x-y|\}_{y\in X}$ as a family of (nonlinear) projections. One can then ask if it satisfies an estimate similar to that of Theorem \ref{thm:Bourgain-projection}. It turns out that it does:
\begin{theorem}[{\cite[Theorem 1.1]{Shmerkin20}}]  \label{thm:nonlinear-distance}
Given $\alpha\in (0,2)$, $\beta>0$, there is $\eta=\eta(\alpha,\beta)>0$ such that the following holds: let $X\subset\R^2$ be a Borel set with $\hdim(X)\ge\alpha$. Then
\begin{equation} \label{eq:nonlinear-distance}
\hdim(\mathcal{E}(X,\eta)) \le \beta,\quad\text{where } \mathcal{E}(X,\eta)= \{ y\in\R^2: \hdim(\Delta_y(X))<\tfrac{\alpha}{2}+\eta\}.
\end{equation}
\end{theorem}
In particular, taking $\beta<1=\alpha$, this provides a pinned version of Theorem \ref{thm:Bourgain-distance}.

Theorem \ref{thm:nonlinear-distance} follows from a general scheme that can be seen as a nolinear extension and refinement of Bourgain's projection theorem and its higher rank generalization by W. He. See \cite{Shmerkin20} for further discussion and precise statements. This scheme yields Theorem \ref{thm:nonlinear-distance} also for smooth norms of everywhere positive Gaussian curvature and $\ell^p$ norms for $p\in (1,\infty)$, as well as some partial extensions to higher dimensions; see \cite[Theorem 1.1]{Shmerkin20}.

Using the nonlinear adaptation of Bourgain's projection theorem (along with many other ideas), O. Raz and J. Zahl \cite{RazZahl21}  have recently obtained a further refinement of Theorem \ref{thm:nonlinear-distance}. They show that for every $\alpha\in (0,2)$ there is $\eta=\eta(\alpha)>0$ such that the set $\mathcal{E}(X,\eta)$ from \eqref{eq:nonlinear-distance} is \emph{flat}, which roughly means that it is contained in a union of a small set of lines, see \cite[Definition 1.4]{RazZahl21}. This is optimal since they also observe that Theorem \ref{thm:nonlinear-distance} is sharp in the sense that $\eta\to 0$ as $\beta\to 0$, but the sets that witness this are contained in a line (or a union of a small family of lines). Raz and Zahl also obtain a related single-scale distance set estimate involving only three non-collinear vantage points:
\begin{theorem}[{\cite[Theorem 1.9]{RazZahl21}}]
Given $\alpha\in (0,2)$ there is $\eta=\eta(\alpha)>0$ such that if $E\subset[0,1]^2$ satisfies the non-concentration estimate
\[
|E\cap B(p,r)|_\delta \le \delta^{-\eta} r^\alpha |E|_\delta,\quad p\in [0,1]^2, r\in [\delta,1],
\]
and $y_1,y_2,y_3\in [0,1]^2$ span a triangle of area $\ge \delta^\eta$, then $\max_{i=1}^3 |\Delta_{y_i}X|_{\delta}\ge \delta^{-\alpha/2-\eta}$.
\end{theorem}
Note that the quantitative non-collinearity hypothesis prevents the train-track almost counterexamples discussed above. These are just special cases of general theorems involving nonlinear projections and Blaschke curvature, see \cite{RazZahl21} for further details.

\subsection{Explicit estimates and sets of equal Hausdorff and packing dimension}

The improvements upon the natural threshold $1/2$ that we have seen so far all involve a tiny and unknown parameter $\eta$. The following were the first explicit bounds in the near critical regime:
\begin{theorem}[T.Keleti and P. Shmerkin, \cite{KeletiShmerkin19}]  \label{thm:KeletiShmerkin}
Let $E\subset\R^2$ be a Borel set with $\hdim(E)>1$. Then $\hdim(\Delta(E))>37/54\approx 0.685$, and there is $y\in E$ such that $\hdim(\Delta_y(E))>2/3$ and $\pdim(\Delta_y(E))=(2+\sqrt{3})/4\approx 0.933$.
\end{theorem}
Here $\pdim$ is packing dimension; we refer to \cite[\S 5.9--5.10]{Mattila95} for its definition and basic properties, and recall only that it lies between Hausdorff and Minkowski (box) dimensions. See \cite{Liu19, Shmerkin21b} for some further improvements, always assuming $\hdim(E)>1$.

The first explicit estimates in the critical case $\hdim(E)=d/2$ were obtained only very recently by the author and H. Wang \cite{ShmerkinWang21}:
\begin{theorem} \label{thm:ShmerkinWang-general}
Let $E\subset\R^d$ be a Borel set with $\hdim(E)=d/2$ where $d=2$ or $3$. Then $\sup_{y\in E} \hdim(\Delta_y E) \ge \alpha_d$, where $\alpha_2=(\sqrt{5}-1)/2\approx 0.618$ and $\alpha_3=9/16=0.5625$.
\end{theorem}

While these are the best currently known estimates for general Borel sets, for sets of equal Hausdorff and packing dimension we are able to prove the full strength of Falconer's conjecture:
\begin{theorem} \label{thm:ShmerkinWang-Ahlfors}
Let $E\subset\R^d$, $d\ge 2$, be a Borel set with $\hdim(E)=\pdim(E)=d/2$. Then $\sup_{y\in E}\hdim(\Delta_y E)=1$, and if $E$ has positive $d/2$-dimensional Hausdorff measure then the supremum is attained.
\end{theorem}
If $\hdim(E)=\pdim(E)=\alpha$, then for each $\e>0$ there is $\mu\in\PP(E)$ such that
\[
r^{\alpha+\e} \lesssim_{\e} \mu(B(x,r)) \lesssim_{\e} r^{\alpha-\e}, \quad r\in (0,1], x\in \supp(\mu).
\]
Thus we can interpret this condition as a rough or approximate version of Ahlfors regularity (which corresponds to the case $\e=0$).

In the plane, Theorem \ref{thm:ShmerkinWang-Ahlfors} has several predecessors. In an influential article, Orponen \cite{Orponen17} proved that if $E$ is Ahlfors regular of dimension $1$, then the \emph{packing} dimension of $\Delta(E)$ is $1$. In \cite{Shmerkin19b}, assuming that $\hdim(E)>1$, we showed that there is $y\in E$ with $\hdim(\Delta_y E)=1$; this result was recovered and made more quantitative in \cite{KeletiShmerkin19}. Extending the proof to the critical case $\hdim(E)=1$ and to higher dimensions required new ideas; we sketch some of them in \S\ref{subsec:kaufman-bootstrapping}.

\subsection{A multiscale formula for the entropy of projections}
\label{subsec:multiscale-entropy}

A common theme through the proofs of Theorems \ref{thm:nonlinear-distance}, \ref{thm:KeletiShmerkin}, \ref{thm:ShmerkinWang-general} and \ref{thm:ShmerkinWang-Ahlfors} is the use of a lower bound for the entropy of projections in terms of multiscale decompositions. Recall that the Shannon entropy of $\mu\in\PP(\R^d)$ with respect to a partition $\mathcal{A}$ of $\R^d$ is
\[
H(\mu;\mathcal{A}) = \sum_{A\in\mathcal{A}} \mu(A)\log(1/\mu(A)).
\]
This quantity measures how uniform the measure $\mu$ is among the atoms $A\in\mathcal{A}$. A basic property is that $H(\mu;\mathcal{A})\le \log|\mathcal{A}|$, and in particular
\begin{equation} \label{eq:entropy-dyadic}
H_\delta(\mu):=H(\mu;\DD_\delta) \le \log|\supp\mu|_{\delta}.
\end{equation}
Given a measure $\mu$ and a set $X$ with $\mu(X)>0$, we denote $\mu_X = \tfrac{1}{\mu(X)}\mu|_X\in\PP(X)$. Finally, fix a $C^2$ map $F:U\supset [0,1]^d\to\R$ with no singular points, and let $\theta(x)=\nabla F(x)/|\nabla F(x)|$.
\begin{prop} [{\cite[Proposition A.1]{Shmerkin20}}] \label{prop:entropy-of-image-measure}
Let $\mu\in\PP([0,1)^d)$, let $1>\delta_0>\delta_1>\ldots>\delta_J=\delta$ be a sequence with $\delta_j^2 \le \delta_{j+1}$, $0\le j<J$. Let $F$ be as above. Then, denoting orthogonal projection in direction $\theta$ by $P_\theta(x)=\langle x, \theta\rangle$,
\begin{equation} \label{eq:lower-bound-entropy}
H_\delta(F\mu) \ge - C_{F,d} J   + \int \sum_{j=1}^{J}  H_{\delta_{j+1}}\left(P_{\theta(x)}\mu_{\DD_{\delta_j}(x)}\right) \,d\mu(x).
\end{equation}
\end{prop}
A \emph{local} variant of this formula is a key element in the proof of Furstenberg's sumset conjecture in \cite{HochmanShmerkin12}. Orponen \cite{Orponen17} introduced this approach to the distance set problem. The method was further refined in \cite{Shmerkin19b,KeletiShmerkin19} - these papers highlighted the importance of choosing the scales $\delta_j$ depending on the combinatorics of the measure $\mu$, a point to which we will come back shortly. Thanks to \eqref{eq:entropy-dyadic}, the formula \eqref{eq:lower-bound-entropy} provides a lower bound on box-counting numbers $|F(\supp\mu)|_\delta$. In order to obtain Hausdorff dimension estimates, one needs a more robust (and technical) variant; we refer to \cite[Appendix A]{Shmerkin20} for details and here we stick with \eqref{eq:lower-bound-entropy} for simplicity.

In all our applications the number of scales $J$ is bounded while $\delta\to 0$, and so the error term is negligible. Note that if $Q\in\DD_{\delta_j}$, then $P_\theta Q$ is an interval of length $\lesssim_d \delta_j$, and hence $H_{\delta_{j+1}}(P_\theta\mu_Q)\le \log(\delta_j/\delta_{j+1})+C_d$.

Why is Proposition \ref{prop:entropy-of-image-measure} useful? A key feature is that it linearizes the nonlinear projection $F$; the hypothesis $\delta_j^2 \le \delta_{j+1}$ comes from linearization, and can be dropped if $F$ is linear. Another advantage is that it replaces the single projection $F\mu$ by an average of projections, taken over $x$ and, crucially, over the scales $(\delta_j)$.

\subsection{Theorems for radial and linear projections, and choice of scales}
\label{subsec:radial}

We sketch how Proposition \ref{prop:entropy-of-image-measure} is used to prove the bound $\hdim(\Delta^y E)> 2/3$ from Theorem \ref{thm:KeletiShmerkin}. By Frostman's Lemma \cite[Theorem 8.8]{Mattila95}, there are $\mu,\nu\in\PP(E)$ with
\[
\mu(B(x,r)),\nu(B(y,r))\lesssim r^\alpha, \quad x,y\in\R^2, r>0,
\]
where $\alpha>1$, and $X=\supp(\mu)$, $Y=\supp(\nu)$ are disjoint. We apply  Proposition \ref{prop:entropy-of-image-measure} to the family $(\Delta_y)_{y\in Y}$ and $\mu$. Since $\nabla\Delta_y(x)=\pi_x(y):=|y-x|/(y-x)$, Equation \eqref{eq:lower-bound-entropy} becomes
\begin{equation} \label{eq:entropy-of-dist-measure}
H_\delta(\Delta_y\mu) \ge - C J   + \int \sum_{j=1}^{J}  H_{\delta_{j+1}}\left(P_{\pi_x(y)}\mu_{\DD_{\delta_j}(x)}\right) \,d\mu(x).
\end{equation}
The scales $\delta_j$ will eventually be chosen in such a way that $\delta_{j+1} \le \delta^c \delta_j$ for some small constant $c$; in particular, this ensures that $J\le \lceil c^{-1}\rceil$ is bounded as $\delta\to 0$.

A radial projection theorem of Orponen \cite{Orponen19} yields $\int \|\pi_x\nu\|_p^p d\mu(x)<\infty$ for some $p=p(\alpha)>1$; it is here that the hypothesis $\hdim(E)>1$ gets used. Restricting $\mu$, we may thus assume that $\|\pi_x\nu\|_p \lesssim 1$ for all $x\in X$. H\"{o}lder's inequality and a quantitative form of Marstrand's projection theorem \cite[Theorem 9.7]{Mattila95} yield that, for any $x\in X$ and $\rho\in\PP(\R^2)$,
\begin{equation} \label{eq:Marstrand}
\pi_x\nu\left\{ \theta\in S^1: \| P_\theta\rho\|_2^2 \ge \delta^{-\e}I_1(\rho) \right\} \le 2^{c_p\e},
\end{equation}
where $I_1(\rho)=\int |x-y|^{-1}d\rho(x)d\rho(y)$ is the $1$-energy of $\rho$. Once the scales $\delta_j$ are fixed, we write $\mu_{x,j}$ for the convolution of $\mu_{\DD_{\delta_j}(x)}$ with a bump function at scale $\delta_{j+1}$, scaled up by a factor $\delta_j^{-1}$. Applying \eqref{eq:Marstrand} to $\rho=\mu_{x,j}$ for fixed $x$ and $1\le j\le J$, using that $J$ is bounded, and then Fubini, we eventually obtain a point $y\in Y$ and a set $X'\subset X$ with $\mu(X')\ge 1/2$, such that
\[
\|P_{\pi_x(y)}\mu_{x,j}\|_2^2 \le \delta^{-\e} I_1(\mu_{x,j}),\quad  x\in X', 1\le j \le J.
\]
Jensen's inequality can be used to bound
\[
H_{\delta_{j+1}}(P_{\pi_x(y)}\mu_{x,j}) \ge \log(\delta_j/\delta_{j+1}) - \log\|P_{\pi_x(y)}\mu_{x,j}\|_2^2 - C.
\]
Putting everything together, \eqref{eq:entropy-of-dist-measure} becomes (denoting a negligible error term by $\text{err}$)
\begin{equation}  \label{eq:final-entropy-bound}
H_\delta(\Delta_y\mu) \ge \log(1/\delta) - \sum_{j=1}^J \int_{X'} \log(I_1(\mu_{x,j})) \,d\mu(x) - \text{err}.
\end{equation}
Now the task has become to choose the scales $\delta_j$ (depending on $\mu$!) subject to the constrains $\delta_j^{1/2}\le \delta_{j+1} \le \delta^c \delta_j$, in such a way that $\log(I_1(\mu_{x,j}))$ is minimized on average. This is a combinatorial problem that becomes more tractable by first uniformizing $\mu$ by applying (the measure version of) Lemma \ref{lem:uniformization} and using the branching numbers $N_j$ as the combinatorial input. The issue to deal with is that, even though $\mu$ is $\alpha$-dimensional, many of the measures $\mu_{x,j}$ can be nearly atomic (if $N_j\approx 1$), which causes the $1$-energy to explode, so one seeks to merge the scales at which this happens with coarser scales at which $\mu$ looks like a large dimensional set. The value $2/3$ is the outcome of this combinatorial problem. Note that if $\mu$ is (roughly) Ahlfors regular, then so are the measures $\mu_{x,j}$, and then $\log(I_1(\mu_{x,j}))$ is uniformly small - so \eqref{eq:final-entropy-bound} also yields $\hdim(\Delta^y E)=1$ in this case.

We underline that even though linearization is at the core of this approach, it is still crucial that the distance map is \emph{nonlinear}, as this is what generates a rich set of directions $\pi_x(y)=\frac{d}{dx}\Delta_y(x)$ to work with  - curvature is still key!

The proofs of Theorems \ref{thm:Bourgain-distance}, \ref{thm:ShmerkinWang-general} and \ref{thm:ShmerkinWang-Ahlfors} follow a similar approach, but they each involve different radial and linear projection theorems. For example, Theorem \ref{thm:Bourgain-distance} relies (unsurprisingly) on Theorem \ref{thm:Bourgain-projection} and a \emph{different} radial projection bound of Orponen \cite{Orponen19}. One feature of Theorems \ref{thm:ShmerkinWang-general} and \ref{thm:ShmerkinWang-Ahlfors} is that they depend on \emph{new} radial and linear projection theorems; we briefly describe them in the next section, in the planar case.

\subsection{Improving Kaufman's projection Theorem, and radial projections}
\label{subsec:kaufman-bootstrapping}
Let $E\subset\R^2$ be as in Theorems \ref{thm:ShmerkinWang-general} or \ref{thm:ShmerkinWang-Ahlfors}. Fix $\e>0$ and, as before, let $\mu,\nu$ be Frostman measures on $E$ with exponents $1-\e$,  and disjoint supports $X, Y$.  If either $\mu$ or $\nu$ gives positive mass to a line, then $E$ intersects that line in dimension $\ge 1-\e$, which makes the distance set estimate immediate. So we may assume that $\mu,\nu$ give zero mass to all lines.

As our discussion in \S\ref{subsec:radial} suggests, it is key to understand radial projections $(\pi_y)_{y\in Y}$ first, and for this we use (again)
 Proposition \ref{prop:entropy-of-image-measure}. Note that $\tfrac{d}{dx}\pi_y(x) = \pi_y(x)^\perp$: this means that in order to estimate radial projections in this way, we need to rely on \emph{a priori} radial projection bounds. This opens the door to bootstrapping arguments, and this is exactly what is done to prove Theorems \ref{thm:ShmerkinWang-general} and \ref{thm:ShmerkinWang-Ahlfors}.

To start the bootstrapping, we need an a priori radial projection estimate for measures of dimension $\le 1$ that give zero mass to lines (note that if both measures are supported on the same line, the radial projections $\pi_x\nu$ are atomic for all $x\in X$; this is why we excluded this case at the beginning of the argument). This is provided by a result of Orponen from \cite{Orponen19}, that we alluded to earlier in connection with Theorem \ref{thm:Bourgain-distance}. In our setting it asserts that for a set of $x$ of $\mu$-measure $\ge 1/2$, the radial projection $\pi_x\nu$ satisfies a Frostman condition of exponent $1/2-\e$ (more precisely this holds after restricting $\nu$ further, depending on $x$).

The goal is to apply  Proposition \ref{prop:entropy-of-image-measure} to $(\pi_y)_{y\in Y}$ to bootstrap the parameter $1/2-\e$ to $1$ and to $(\sqrt{5}-1)/2$ in the Ahlfors regular and general case, respectively. Following the scheme of \S\ref{subsec:radial}, we end up needing a certain linear projection theorem, that we discuss next. A classical projection theorem of R.Kaufman \cite{Kaufman68} from 1968 asserts that if $X\subset\R^2$ is a Borel set and $s<\min(1,\hdim(X))$, then
\[
\hdim\{\theta\in S^1: \hdim(P_\theta X)\le s\} \le s.
\]
It is natural to conjecture that Kaufman's theorem is not optimal, in that the bound $s$ on the right-hand side can be lowered, depending on $s$ and $\hdim(X)$.  When $s\le \hdim(X)/2+\eta(\hdim(X))$, such improvement follows from Theorem \ref{thm:Bourgain-projection-dim}, but the general case was established only very recently by T. Orponen and the author:
\begin{theorem}[{\cite[Theorem 1.2]{OrponenShmerkin21}}] \label{thm:Kaufman-improvement}
Given $s\in (0,1)$, $t\in (s,2]$ there is $\e=\e(s,t)>0$ such that if $X\subset\R^2$ is a Borel set with $\hdim(X)\ge t$, then
\[
\hdim\{\theta\in S^1: \hdim(P_\theta X)\le s\} \le s-\e.
\]
\end{theorem}
The proof uses many of the ingredients we have discussed in this survey: Bourgain's projection theorem, the uniformization lemma, and choosing the scales depending on the given measure. There are also new ideas, including an ``incidence version'' of Proposition \ref{prop:entropy-of-image-measure} and a dichotomy between the ``roughly Ahlfors regular'' and ``far from Ahlfors regular'' situations, each requiring  different arguments.

A quantitative version of Theorem \ref{thm:Kaufman-improvement} (\cite[Theorem 1.3]{OrponenShmerkin21}) provides the input necessary to complete the bootstrapping step in the proofs of the planar cases of Theorems \ref{thm:ShmerkinWang-general} and \ref{thm:ShmerkinWang-Ahlfors}. To be more precise, so far we have been considering radial projections, but because $\tfrac{d}{dx}\pi_y$ and $\tfrac{d}{dx}\Delta_y$ are rotations of each other, the argument for distance sets can be completed in parallel. The golden mean $(\sqrt{5}-1)/2$ arises as the outcome of the combinatorial problem of optimizing the choice of scales (after uniformization).


\section*{acknowledgements}
Thanks to P\'{e}ter Varj\'{u}, Hong Wang and Josh Zahl for comments and corrections on earlier versions of the manuscript.


\end{document}